\documentclass[12pt]{article}

\usepackage[T1]{fontenc}          
\usepackage[utf8]{inputenc}

\usepackage{amsmath,amsfonts,amssymb,amsthm,mathtools}
\usepackage{mathrsfs}             
\usepackage{bm}

\usepackage[a4paper]{geometry}
\usepackage{graphicx}
\usepackage{float}
\usepackage{placeins}
\usepackage{caption}
\usepackage{marginnote}
\usepackage{comment}

\usepackage{array}
\usepackage{multirow}
\usepackage{booktabs}
\usepackage{longtable}
\usepackage{tabularx}
\usepackage{needspace}

\usepackage[shortlabels]{enumitem}

\usepackage{algorithm}
\usepackage{algpseudocode}

\usepackage{tikz}
\usetikzlibrary{positioning}
\usetikzlibrary{calc}

\usepackage[unicode]{hyperref}
\hypersetup{
	colorlinks=true,
	linkcolor=blue,
	citecolor=blue,
	urlcolor=blue
}

\numberwithin{figure}{section}
\numberwithin{equation}{section}

\makeatletter

\newtheoremstyle{myplain}
{6pt}{6pt}%
{\itshape}
{}%
{\bfseries}
{.}%
{0.5em}%
{}%

\newtheoremstyle{mydefinition}
{6pt}{6pt}%
{\normalfont}
{}%
{\bfseries}
{.}%
{0.5em}%
{}%

\newtheoremstyle{myremark}
{6pt}{6pt}%
{\normalfont}
{}%
{\bfseries}
{.}%
{0.5em}%
{}%

\renewenvironment{proof}[1][Proof]{%
	\par\pushQED{\qed}\normalfont
	\topsep6pt\relax
	\trivlist
	\item[\hskip\labelsep\bfseries #1.]%
}{%
	\popQED\endtrivlist\@endpefalse
}

\makeatother

\theoremstyle{myplain}
\newtheorem{theorem}{Theorem}[section]

\theoremstyle{mydefinition}

\theoremstyle{myremark}

\theoremstyle{myplain}
\newtheorem{theo}[theorem]{Theorem}
\newtheorem{statm}[theorem]{Statement}
\newtheorem{cor}[theorem]{Corollary}

\theoremstyle{mydefinition}
\newtheorem{defi}[theorem]{Definition}

\theoremstyle{myremark}
\newtheorem{rmrk}[theorem]{Remark}

\newcommand{\df}{\operatorname{df}}       
\newcommand{\Ry}{\operatorname{Ry}}       
\newcommand{\capa}{\operatorname{cap}}    
\newcommand{\vcount}{\operatorname{vcount}}

\begin{document}
\title{Distribution of rooks on a chess-board representing a Latin square partitioned by a subsystem (Part 1.)\footnote{We intend to publish several consecutive papers on Latin and partial Latin squares and on bipartite graphs, so it is worth studying them in the order of publishing.}
}
\author{Béla Jónás}

\date{Revised July 2026 \hspace{3pt} Budapest, Hungary}

\maketitle

\begin{abstract}
A $d$-dimensional generalization of a Latin square of order $n$ is a chess-board of size
$n\times n\times\ldots\times n$ ($d$ times) carrying $n^{d-1}$ non-attacking rooks,
equivalently a high-dimensional permutation in the sense of Linial and Luria. For a
subsystem $T$ induced by a family $\langle E_1,\ldots ,E_d\rangle$ of subsets of
$\{1,2,\ldots ,n\}$, let $\df(T)=V(T)/n-c$, where $V(T)$ is the number of cells of $T$ and
$c$ the number of rooks in it.

Replacing $k$ of the sets $E_i$ by their complements yields a subsystem $T_k$, and the
$2^d$ subsystems so obtained partition the chess-board. We prove that
\[
\df(T_k)=(-1)^k\df(T_0),
\]
and give a multilinear proof showing that no assumption on the sets $E_i$ is needed and
that the result holds verbatim for $d$-fold stochastic matrices. Consequently the partition
has a single degree of freedom: the rook counts of all its $2^d$ members follow from their
volumes and the single number $\df(T_0)$. For $d=k=3$ this specializes to the identity of
Cruse relating a brick and its remote mate, a necessary condition for a partial Latin square
to be completable. We also prove that for $d=3$ the chess-board represents precisely one
main class.
\end{abstract}

\textbf{MSC-Class:} 05B15


\section*{Abbreviations}
\begin{tabular}{@{}ll}
LS&Latin square\\
LSC&Latin Super Cube\\
$d$-LSC&$d$-dimensional Latin Super Cube\\
RBC&remote brick couple
\end{tabular}

\section{Introduction}

The \emph{Hamming distance} between two \mbox{$d$-tuples} is the number of positions at which the corresponding coordinates differ. Formally if $a=(a_1,\ldots,a_d)$ and $b=(b_1,\ldots,b_d)$ then $d(a,b)=\left|\{i\colon a_i\neq b_i\}\right|$.

A \emph{Latin square of order $n$} is an $n\times n$ array consisting of $n$ different symbols such that each symbol appears exactly once in each row and column. The property that no row or column contains any symbol more than once is known as the \emph{Latin property}.

From now on, the set of symbols is always $\mathbb{Z}_n=\{1,2,\ldots,n\}$, so the Latin square can also be considered as a set of ordered triples of the form $(i,j,k)$, where $i$, $j$, $k\in\mathbb{Z}_n$, and each Hamming distance between two distinct triples is at least~2.

Two Latin squares are \emph{isotopic} if each can be obtained from the other by permuting the rows, columns, and symbols. This is an equivalence relation, with the equivalence classes are called \emph{isotopy classes}.

Each Latin square $Q$ has six \emph{conjugate} Latin squares obtained by uniformly permuting the coordinates in each of its triples. They are denoted by $Q(i,j,k)$, $Q(j,k,i)$, $Q(k,i,j)$, $Q(j,i,k)$, $Q(i,k,j)$, $Q(k,j,i)$, where $Q(i,j,k)=Q$.

Two Latin squares are called \emph{paratopic}, if one of them is isotopic to a conjugate of the other. This is an equivalence relation, where the equivalence classes are called \emph{main classes}.

\section{d-LSC}
A $3$-dimensional chess-board of size $n\times n\times n$ is denoted by $H_n^3$ or simply $H^3$.
Each cell is identified by a triple $(i,j,k)$, where $i$, $j$, $k\in\mathbb{Z}_n$ and the distance between two cells is the Hamming distance.
The chess-board $H_n^3$ can be regarded as a Rubik's cube with $n^3$ cells (cubelets) and can be placed in a coordinate system with one vertex at the origin, as depicted in  Figure~\ref{fig1_1}.
\begin{figure}[htb]
\centering\includegraphics[scale=0.3] {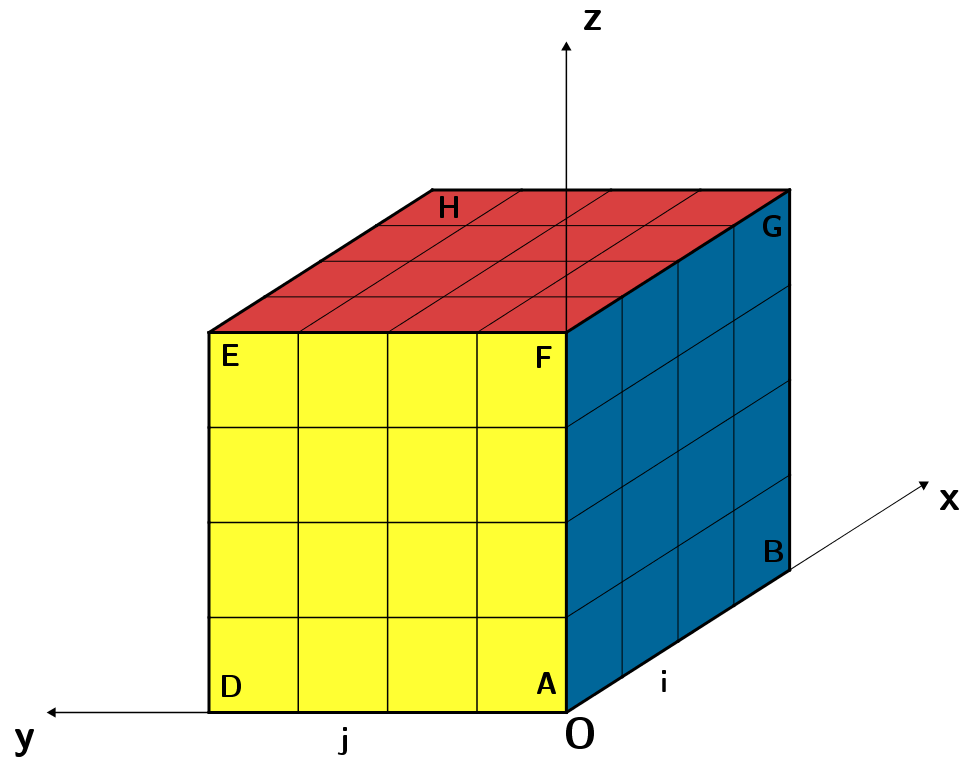}
\caption{}\label{fig1_1}
\end{figure}

There is an immediate generalization of this structure to dimension $d$, where $d > 3$. The chess-board $H_n^d$ contains $n^d$ cells and each cell is identified by a \mbox{$d$-tuple} $(e_1,e_2,\ldots,e_d)$ where $e_i\in\mathbb{Z}_n$ for each $i \in\mathbb{Z}_d$. In case of $d=2$ and $d=3$ we generally denote the coordinate axes $x$, $y$ and $x$, $y$, $z$, respectively. 
An illustration of the latter case is shown in Figure ~\ref{fig1_1}. 

If $d>3$, then the axes of the chess-board $H_n^d$ are denoted by $t_1$, $t_2$, \ldots, $t_d$. 
If a cell is identified by a \mbox{$d$-tuple} $(e_1,e_2,\ldots,e_d)$, then we consider $e_i$ as a coordinate of the cell on the axis $t_i$ for each $i\in \mathbb{Z}_d$.

\begin{defi}
Let $H^d_n$ be a $d$-dimensional chess-board and $j$ be a fixed coordinate on the axis $t_i$, where $j\in \mathbb{Z}_n$ and $i\in \mathbb{Z}_d$. The set of cells of the chess-board whose coordinate is $j$ on the axis $t_i$, is called the \emph{$(d-1)$-dimensional subspace} of $H^d_n$ and denoted by $H_i(j)$. 
\end{defi}

\begin{cor}
For any axis $t_i$, the subspaces $H_i(j)$ form a partition of $H^d_n$ if $j$ goes from $1$ to $n$.
\end{cor}

\begin{defi}
A subspace of the chess-board of dimension 1 is called a \emph{file}, of dimension 2 is called a \emph{layer}.
\end{defi}

\begin{defi}
If the chess-board $H_n^d$ contains exactly $n^{d-1}$ non-attacking rooks, then the structure is called a \emph{Latin Super Cube of dimension d} or a \emph{$d$-LSC} for short.
\end{defi}
\begin{rmrk}\label{rmrk:hdperm}
The $d$-LSC is not a new object. Linial and Luria~\cite{[3]} call it a
\emph{high-dimensional permutation}: for them a \emph{$k$-dimensional permutation of
order $n$} is an $[n]^{k+1}$ array of zeros and ones in which every line contains a
unique~$1$, a line being the set of entries obtained by letting one coordinate run
over $\mathbb{Z}_n$ and fixing all the others. Their $1$-dimensional permutation is a
permutation matrix and their $2$-dimensional permutation is an order-$n$ Latin square.
Mind the shift of the index: a $d$-LSC is a $(d-1)$-dimensional permutation in their
sense, so our $2$-LSC and $3$-LSC are their $1$- and $2$-dimensional permutations.
\end{rmrk}

A brief description of the concept can be found in~\cite{[1]}. Instead of the term “Latin property” we often say that the rooks are non-attacking, or they \emph{do not see each other}. Each subspace of the chess-board of dimension $k$, where $k\in \mathbb{Z}_d$, is itself considered as a $k$-LSC with $n^k$ cells and $n^{k-1}$ non-attacking rooks. Consequently, a file contains 1 rook, a layer contains $n$ non-attacking rooks.

\section{One Cube Represents One Main Class}
The first element of a double identifying the cell $(i,j)$ of a Latin square is the row coordinate on the $x$-axis, the second is the column coordinate on the $y$-axis and the plane containing the both axes is denoted by $[xy]$. The cube, in which rooks are placed according to symbols, lies on the plane and one vertex, denoted by $A$, is at the origin of $[xy]$ and the two neighboring vertices, denoted by $B$ and $D$ are on the $x$-axis and $y$-axis, respectively. 
The third neighboring vertex of $A$, denoted by $F$, is on the $z$-axis.
\begin{defi}
Building a $3$-LSC from a Latin square by placing a rook into the cublet $(i,j,k)$ when the cell $(i,j)$ of the Latin square contains a symbol $k$ is called \emph{composition}. 
\end{defi}
This implies that a 3-LSC derived by composition from a Latin square is always created in a right-handed coordinate system and the symbols no longer play a special role in the LSC (a symbol value is just one of the three coordinates).
\begin{defi}
Place a 3-LSC in a right-handed coordinate system with one vertex at the origin. Creating a Latin square by placing a symbol $k$ into the cell $(i,j)$ of the square when the cublet $(i,j,k)$ contains a rook is called \emph{projection}.
\end{defi}

Since the rooks are non-attacking, the result of the projection of a \mbox{3-LSC} to any face of the cube is clearly a Latin square. 

\begin{defi}
Two 3-LSCs are \emph{isotopic} if each can be obtained from the other by permuting the rows layers, columns layers and symbols layers.
\end{defi}

\begin{rmrk}\label{rmrk240}
Let $L$ be a $3$-LSC derived from a Latin square $Q$ by composition. There is a one-to-one correspondence between the permutations of the rows, columns and symbols of $Q$ and the permutations of the row, column and symbol layers of $L$. Therefore, there is a one-to-one correspondence between the isotopes of $Q$ and isotopes of $L$, that is if $\mathcal{P}$ is a series of permutations that transforms $Q$ to $Q^*$ then applying $\mathcal{P}$ on $L$, the resulting LSC $L^*$ can be derived from $Q^*$ by composition.
\end{rmrk}

\begin{defi}
Since the permutations of the layers and the rotations of the cube leave the property unchanged that the rooks are non-attacking, they are called \emph{Latin transformations}.
\end{defi}

If we consider a $3$-LSC $L$ derived from a Latin square $Q$ by composition the $L$ does not ``remember'' which face of the cube was on the plane when the composition was applied, so applying Latin transformations, each face of the cube can be the face to apply a projection. To apply a projection to a specific face of the cube, without losing generality, we always rotate the cube so that the given face is on the plane [xy] and a pre-selected vertex is at the origin.

\begin{defi}
A face of the cube can be identified by 3 adjacent vertices, so the face that contains vertices $A$, $B$ and $D$ is denoted by $[BAD]$ where the middle vertex $A$ is at the origin, the first vertex $B$ is on the $x$-axis and the third vertex $D$ is on the $y$-axis, as shown in the Figure~\ref{fig1_1}. The face $[BAD]$ is called \emph{bottom face}, and the parallel face is called \emph{cover face} and can be denoted $[EFG]$ or $[GHE]$.
\end{defi}

\begin{rmrk}\label{rmrk270}
Note that in a right-handed coordinate system with the face $[EFG]$ on the plane [xy], the edge $EF$ is on the $x$-axis and the edge $FG$ is on the $y$-axis and similarly with the face $[GHE]$, the edge $GH$ is on the $x$-axis and the edge $HE$ is on the $y$-axis, since the cube is located under the given faces, so the coordinates on the z-axis increase downwards.
\end{rmrk}

\begin{defi}
Let $Q[BAD]$ denote the Latin square derived from the $3$-LSC $L$ by projection to the bottom face $[BAD]$.
\end{defi}

\begin{theo} 
If $L$ is a 3-LSC with $n^2$ non-attacking rooks and $Q$ is the Latin square derived from $L$ by projection to an arbitrary face of the cube then the LSCs obtained from $L$ by Latin transformations are all and only those LSCs, that can be derived from a paratope of $Q$ by composition.  
\end{theo}

\begin{proof}
Let $L$ be a 3-LSC in a right-handed coordinate system with vertex $A$ at the origin, $B$ on the $x$-axis, $D$ on the $y$-axis and $F$ on the $z$-axis as indicated in  Figure~\ref{fig1_1}. Let $Q$ be the Latin square on the face $BAD$ derived from $L$ by projection, hence $Q[BAD]=Q=Q(i,j,k)$, according to our notations. 

Derive $Q[DAF]$ from $L$ by projection onto the face $[DAF]$ in the right-handed coordinate system $(y,z,x)$, the rook in the cell, originally identified by 3-tuple, $(i,j,k)$ results the symbol $i$ in the cell $(j,k)$, so $Q[DAF] = Q(j,k,i)$. 
In the same way, in the right-handed coordinate system $(z,x,y)$, the rook in the cell $(i,j,k)$ projected onto the face $[FAB]$ results the symbol $j$ in the cell $(k,i)$, thus $Q[FAB] = Q(k,i,j)$. So, if we take the Latin squares derived from $L$ by projection onto the 3 faces that coincide at the origin of the cube, we get 3 conjugate Latin squares $Q=Q(i,j,k)$, $Q(j,k,i)$ and $Q(k,i,j)$.
\begin{defi}
The Latin squares $Q(i,j,k)$, $Q(j,k,i)$ and $Q(k,i,j)$, where  $Q(i,j,k) = Q$ are called \emph{primary conjugates of $Q$}.
\end{defi}

Permute the layers on each axis of $L$ in reverse order. The resulting LSC and $L$ are isotopic and the cell $(i,j,k)$ is moved to cell $(n+1-i,n+1-j,n+1-k)$. Place a right-handed coordinate system with the origin at vertex $H$ such that the $z$-axis is aligned with the edge $HC$ of the cube. Then the edge $HG$ is on the $x$-axis and the edge $HE$ is on the $y$-axis. So, the new $z$-axis is parallel to the old one, however the new $x$-axis is parallel to the old $y$-axis and the new $y$-axis is parallel to the old $x$-axis, according to Remark~\ref{rmrk270}. Consequently, the new coordinates of the cell $(n+1-i,n+1-j,n+1-k)$ are $(j,i,k)$, so, the projection to the face $GHE$ gives the Latin square $Q(j,i,k)$, the projection to the face $EHC$ gives the Latin square $Q(i,k,j)$, projection to the face $CHG$ gives the Latin square $Q(k,j,i)$.

\begin{defi}
The Latin squares $Q(j,i,k)$, $Q(i,k,j)$ and $Q(k,j,i)$ are called \emph{secondary conjugates of $Q$}.
\end{defi}

The primary conjugates of $Q$ can be derived from $L$ by projection, the secondary conjugates of $Q$ can be derived from a specific isotope of $L$ by projection. Consequently, based on Remark~\ref{rmrk240} each element of the main class of $Q$ can be derived from the proper isotope of $L$.

Now we prove that a Latin square $Q^*$ derived from an LSC $L^*$ produced from $L$ by Latin transformations is an element of the main class of $Q$.

All 6 faces of the cube can be moved to the bottom face and all four vertices of
the bottom face can be rotated to the origin without changing the bottom face, so
the cube has 24 positions, in accordance with the fact that the rotation group of
a physical cube has 24 elements.

\begin{defi}\label{def:facerot}
A rotation of the cube by $-90$ degrees about the axis that passes through the
centre of the cube and is parallel to one of the coordinate axes $x$, $y$, $z$ is
called a \emph{face rotation}. These three rotations are denoted by $f_x$, $f_y$
and $f_z$, respectively.
\end{defi}

The rotation $f_z$ is the one illustrated in Figure~\ref{fig1_1}: it moves the
vertex $B$ to the origin and the face $ABCD$ remains the bottom face, whereas
$f_x$ and $f_y$ bring a new face into the bottom position. Since $f_x$, $f_y$ and
$f_z$ are rotations about two distinct axes, they generate the whole rotation
group of the cube; in fact any two of them already do. Hence from each position of
the cube all 24 positions can be achieved with a sequence of face rotations.
An equivalence class is transitive, so it is enough to prove that a Latin square
$Q^*$ derived from an LSC $L^*$ produced from $L$ by a \emph{single} face rotation
is an element of the main class of $Q$.

Let $f_a$ be a face rotation, where $a\in\{x,y,z\}$, let $L^*=f_a(L)$ and let $Q^*$
be the Latin square derived from $L^*$ by projection to the bottom face. The three
axes play symmetric roles, since a rook is nothing but a triple $(i,j,k)$ and the
Latin property (any two distinct triples are at Hamming distance at least~$2$) is
symmetric in the three coordinates. Therefore it suffices to describe the effect of
$f_a$ on the two axes it moves, the third coordinate being left unchanged.

Consider first $f_z$. It moves the row layer $i$ of $L$ into the column layer
$(n+1-i)$, and the column layer $j$ into the row layer $j$; the coordinates of the
symbol layers remain unchanged. The order of the row layers does not change, but the
layers are moved to the other side of the origin, so their coordinates change from
$1,2,\ldots ,n$ to $n,n-1,\ldots ,1$, therefore the new column coordinate is
$(n+1-i)$. Consequently
\[
f_z\colon (i,j,k)\mapsto (j,\;n+1-i,\;k),
\]
and by the symmetry of the three axes,
\[
f_x\colon (i,j,k)\mapsto (i,\;k,\;n+1-j),\qquad
f_y\colon (i,j,k)\mapsto (n+1-k,\;j,\;i).
\]

Now permute, in reverse order, the layers perpendicular to the axis on which the
term $(n+1-\cdot)$ appears: the column layers for $f_z$, the symbol layers for
$f_x$ and the row layers for $f_y$. By Remark~\ref{rmrk240} such a permutation of
layers is an isotopy, and it turns the three maps above into
\[
(i,j,k)\mapsto (j,i,k),\qquad
(i,j,k)\mapsto (i,k,j),\qquad
(i,j,k)\mapsto (k,j,i),
\]
respectively. Hence $Q^*$ is isotopic to a conjugate of $Q$ in each of the three
cases, that is, $Q^*$ is a paratope of $Q$.
\end{proof}

\begin{rmrk}
Permuting the layers of $H^3$ perpendicular to a given axis in reverse order corresponds to a plane reflection in the Euclidean sense. Permuting the layers on each axis of $H^3$ in reverse order corresponds to central symmetry. 
\end{rmrk}

\begin{rmrk}
The three maps obtained at the end of the proof are exactly the three transpositions
of the coordinates $(i,j,k)$. Since the transpositions generate the symmetric group
$S_3$, the face rotations, combined with permutations of layers, already produce all
six conjugates of $Q$. In particular, the secondary conjugates arise from the face
rotations $f_x$, $f_y$, $f_z$ themselves, which gives a second derivation of them,
independent of the central symmetry argument used above.
\end{rmrk}

\begin{rmrk}
Secondary conjugates do not give us additional information, so from now on we only deal with primary conjugates.
\end{rmrk}

\goodbreak
\section{Hamming Bricks}

Let $X$ a subsystem induced by a family of sets $<E_1,E_2,\ldots ,E_d>$ over $\mathbb{Z}_n$. Then $X$ contains all cells identified by the \mbox{$d$-tuples} $(e_1,e_2,\ldots ,e_d)$ for which $e_i \in  E_i$ for each $i\in \mathbb{Z}_d$.

\begin{defi}
The subsystem $X$ is visible if $max\{E_i\} - min\{E_i\} = |E_i|-1$
for each $i\in \mathbb{Z}_d$. If $X$ is visible, it is called a \emph{Hamming brick} or simply a \emph{brick}.
\end{defi}

For each set $E_i$, there exists a permutation $p_i$ which places the different elements of the set $E_i$ as coordinates on the axis $t_i$ into the set of coordinates $1, 2,\ldots ,|E_i|$, leaving the coordinates on the other axes unchanged. This permutation only changes the order of the $(d-1)$-dimensional subspaces $H_i(j)$, where $j\in \mathbb{Z}_n$.
After executing all the permutations $p_1,p_2,\ldots ,p_d$, the subsystem $X$ is a brick permuted to the origin. We generally examine the subsystems in this form.

\begin{rmrk}
A brick permuted to the origin of the chess-board is a solid rectangular cuboid of size \mbox{$|E_1|\times |E_2|\times \ldots\times |E_d|$} in the Euclidean geometry.
\end{rmrk}

\begin{defi}
Let $X$ be a set of cells and $c$ be a cell of $H_n^d$. We define the Hamming distance between $X$ and $c$ as follows:
\[
d(X,c)=\min\{d(x,c)\mid x\in X\},
\]
where $d(x,c)$ is the Hamming distance between the cells $x$ and $c$.
\end{defi}

Clearly $d(X,c)>0$ if and only if $c\notin X$.

\begin{defi}
Let $T\subseteq H_n^d$ be a brick. The Hamming sphere of radius $r$, center $T$ is a set of cells for which
\[
S_r(T)=\{c\in H_n^d\mid d(T,c)=r\}.
\]
\end{defi}

Obviously $S_0(T)=T$.

So, we generalized the cell-centered Hamming sphere to a brick-centered Hamming sphere. The Figure ~\ref{fig2_1} shows a Hamming sphere around the cubelet K with $r=1$, a Hamming sphere around the brick T with $r=1$. 

\begin{figure}[htb]
\centering
\begin{tabular}{c}
\hbox to 0.9\textwidth {\includegraphics[scale=.45]{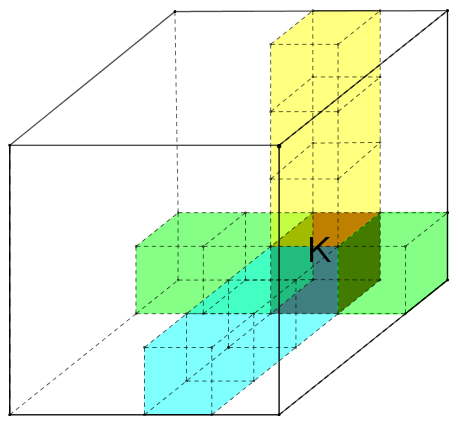}\hfill\includegraphics[scale=.45]{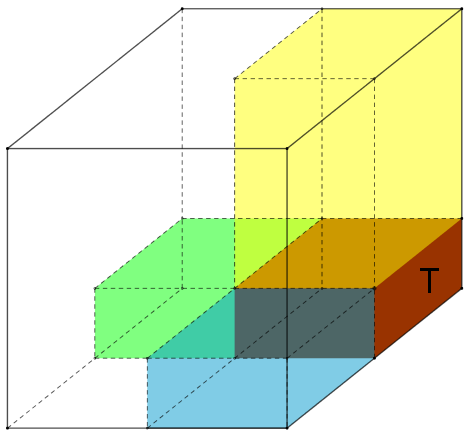}}
\end{tabular}
\caption{}\label{fig2_1}
\end{figure}

The Figure ~\ref{fig2_2} shows a Hamming sphere around the brick T with $r=2$ and a Hamming sphere around the brick T with $r=3$.

\begin{figure}[htb]
\centering
\begin{tabular}{c}
\hbox to 0.9\textwidth {\includegraphics[scale=.45]{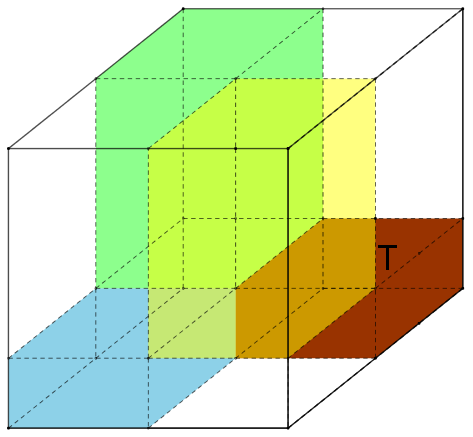}\hfill\includegraphics[scale=.45]{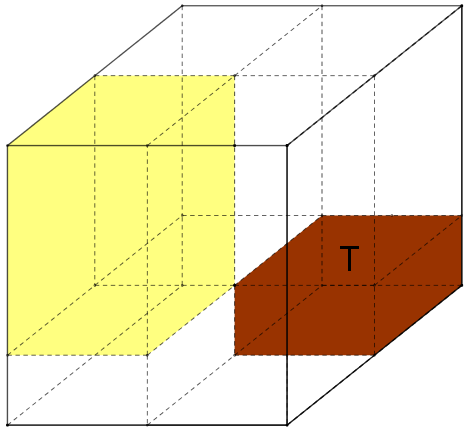}}
\end{tabular}
\caption{}\label{fig2_2}
\end{figure}

\begin{defi}
Let $T$, $U\subseteq H_n^d$ be two bricks. The Hamming distance between these two bricks~is:
\[
d(T,U)=\min\{d(x,y)\mid x\in T, y\in U\} 
\]
\end{defi}

It is evident that $d(T,U)=0$ if and only if $T\cap U\neq\emptyset$.

For $d=3$, $\dbinom{3}{0}$ brick has distance $0$ from brick $T$ ($T$ itself), $\dbinom{3}{1}$ bricks have distance 1 from $T$ and $\dbinom{3}{2}$ bricks have distance 2 from $T$ and $\dbinom{3}{3}$ brick has distance 3 from $T$.

\begin{defi}
A brick of size $e_1\times e_2\times \ldots\times e_d$ is a \emph{real brick}, if $1\leq e_1,e_2,\ldots,e_d<n$.
\end{defi}

Let $T_0$ be a brick of size $e_1\times e_2\times \ldots\times e_d$, where $1\leq e_1,e_2,\ldots,e_k<n$, $e_{k+1}=\ldots=e_d=n$. In this case $k$ is the maximum distance between the brick $T_0$ and a cell, so the maximum distance between $T_0$ and another brick.

Let $I$ be a set of indices, where $I\subseteq\{1,2,\ldots,k\}$ and $r=|I|$. Let us define the set $T_I$ as follows
\[
T_I=\{x=(x_1,x_2,\ldots,x_d)\in H_n^d\mid n\geq x_i>e_i\text{ for }i\in I\text{ and }1\leq x_i\leq e_i\text{ for }i\notin I\}.
\]
Obviously, $T_I$ is a brick as well and $d(T_0,T_I)=r$.

Consequently, we have $\dbinom{k}{r}$ subsets of cardinality $r$, so we have $\dbinom{k}{r}$ distinct disjoint bricks that have distance $r$ from $T_0$. Using the Hamming distance from $T_0$, the bricks $T_I$ generate a partition of $H_n^d$ into $2^k$ disjoint bricks, one for each index set $I\subseteq\{1,2,\ldots,k\}$. If $T_0$ is a real brick, then $k=d$, and each Hamming sphere of radius $r$, with center $T_0$ consists of $\dbinom{d}{r}$ disjoint bricks for $r \in\mathbb{Z}_d$ and $r=0$.

\begin{rmrk}\label{417}
A brick $T_k$ at a distance $k$ from a real brick $T_0$ has exactly $k$ edges that are obtained by taking the edge $(n-e_i)$ instead of the edge $e_i$ of $T_0$.
\end{rmrk}

\begin{rmrk}
Let $T$ be an arbitrary brick in the partition generated by $T_0$. Then $T$ generates the same bricks as $T_0$, so we get the the same partition of the space $H_n^d$.
\end{rmrk}

\begin{rmrk}
Let $T_0$ be a real brick of size $e_1\times e_2\times \ldots\times e_d$ permuted to the origin. If we look at the partition generated by $T_0$ in Euclidean geometry, then each brick has a vertex that coincides with a vertex of the chess-board, different bricks fit on different chess-board vertices and the opposite (the only inner) vertex of each brick is the Euclidean point with coordinates $(e_1,e_2,\ldots ,e_d)$.
\end{rmrk}

\begin{defi}
Let $T,U\subseteq H_n^d$ be two bricks. The bricks $T$ and $U$ are \emph{remote} if $d(T,U)=d$. The pair of $T$ and $U$ is called a \emph{remote couple}, if $U$ contains all the cells that have a distance $d$ from $T$. If $T$ and $U$ form a remote couple, then $T$ is called the \emph{remote mate of $U$} and $U$ is called the \emph{remote mate of $T$}.
\end{defi}

If $T$ and $U$ are remote, then $U$ and $T$ are remote as well. If $T$ and $U$ are remote, then $T\cap U=\emptyset$.

\begin{defi}
Let $X$ be a set of cells. We define $V(X)$, the \emph{volume} of $X$ as the number of cells in $X$.
\end{defi}

If $T\subseteq H_n^d$ is a brick of size $e_1\times e_2\times \ldots\times e_d$, then the volume of $T$ is 
$V(T)=\displaystyle\prod_{i=1}^de_i$.

\begin{defi}\label{def:vcount}
Let $T\subseteq H_n^d$ be the subsystem induced by $\langle E_1,E_2,\ldots ,E_d\rangle$.
The \emph{vertex count} of $T$ is
\[
\vcount(T)=|E_1|+|E_2|+\cdots+|E_d|,
\]
so that for a brick of size $e_1\times e_2\times\ldots\times e_d$ we have
$\vcount(T)=e_1+e_2+\cdots+e_d$.
\end{defi}

\begin{rmrk}\label{rmrk:hypergraph}
The name is explained by the hypergraph reading of the chess-board. Let
$V_1,V_2,\ldots ,V_d$ be disjoint sets of size $n$, the elements of $V_i$ being the
coordinates available on the axis $t_i$, and identify the cell
$(e_1,e_2,\ldots ,e_d)$ with the $d$-element set $\{e_1,e_2,\ldots ,e_d\}$, one vertex
from each $V_i$. Under this identification $H_n^d$ is the complete $d$-partite
$d$-uniform hypergraph on $V_1\cup V_2\cup\cdots\cup V_d$, the rooks are hyperedges,
and the subsystem induced by $\langle E_1,E_2,\ldots ,E_d\rangle$ is the
sub-hypergraph spanned by the vertex set $E_1\cup E_2\cup\cdots\cup E_d$. Thus
$\vcount(T)$ is literally the number of vertices that $T$ spans. For $d=2$ the
hypergraph is a bipartite graph and $\vcount$ is the number of rows plus the number
of columns of the brick.
\end{rmrk}

\begin{defi}
Let $T_0\subseteq H_n^d$ be a brick of size $e_1\times e_2\times \ldots\times e_d$.
The \emph{area of $T_0$ orthogonal to $e_i$ or $t_i$} is
\[
A(T_0,e_i)=A(T_0,t_i)=V(T_0)/e_i=e_1e_2\ldots e_{i-1}e_{i+1}\ldots e_{d-1}e_d.
\]
\end{defi}

\begin{defi}
Let $T_0\subseteq H_n^d$ be a brick of size $e_1\times e_2\times \ldots\times e_d$.
The brick $T_i$, in the partition generated by $T_0$, is the \emph{auxiliary brick} of $T_0$ along $t_i$, if $T_i$ contains all of the cells $(x_1,x_2,\ldots,x_i,\ldots,x_d)$ for which $x_k\le e_k$ if $k\neq i$ and $x_k>e_k$ if $k=i$. $T_0$ and $T_i$ together are called \emph{auxiliaries} along $t_i$.
\end{defi}

Clearly, $d(T_0,T_i)=1$ and $T_0$ has $\dbinom{d}{1}=d$ auxiliary bricks. Figure~\ref{fig2_1} shows the case $d=3$.

\begin{rmrk}
If the bricks $U_1$ and $U_2$ are auxiliaries along an axis $t$, then $U_1$ and $U_2$ are disjoint, so
$V(U_1\cup U_2)=V(U_1)+V(U_2)$ and $U_1\cup U_2$ is a brick with an edge of length $n$ on the axis $t$.
\end{rmrk}

\begin{defi}
A brick $T$ is an \emph{$n$-brick}, if $T$ has at least one edge of length $n$. A brick $T$ is an \emph{$n^2$-brick}, if $T$ has at least two edges of length $n$.
\end{defi}
\begin{rmrk}
A subspace is always an $n$-brick.
\end{rmrk}

Let $L$ be a $d$-LSC and $T\subseteq H_n^d$ be an $n$-brick of size $e_1\times e_2\times \ldots\times e_d$, where $e_i=n$. It is clear from the construction, that each file has exactly one rook. Therefore, the number of rooks in $T$ is the size of the area of $T$ orthogonal to $e_i$, i.e, $V(T)/n$. Based on this observation the following results:

Let $L$ be a $d$-LSC and $T_0$ be a brick of size $e_1\times e_2\times \ldots\times e_k\times e_{k+1}\times \ldots\times e_d$, which has $c_0$ rooks. Let $T_k$ be a brick with distance $k$ from $T_0$ in the partition generated by $T_0$ and contain $c_k$ rooks. Based on Remark~\ref{417} and without loss of generality, we assume that $T_k$ is a brick of size
\[
(n-e_1)\times (n-e_2)\times \ldots\times (n-e_k)\times e_{k+1}\times \ldots\times e_d,
\]
otherwise, we change the order of the axes.

\begin{theo}[Distribution Theorem]\label{theo2.17}
\begin{equation}\label{(201)}
c_k=\frac{V(T_k)-(-1)^kV(T_0)}{n}+(-1)^kc_0=\frac{V(T_k)}{n}-(-1)^k\left[\frac{V(T_0)}{n}-c_0\right]
\end{equation}
\end{theo}

\begin{proof}
Let $T_1$ be the auxiliary brick of $T_0$ along $t_1$ and let $T_1$ contain $c_1$ rooks. Since $T_0\cup T_1$ is an $n$-brick, so $T_0\cup T_1$ has $\dfrac{V(T_1\cup T_0)}{n}=\dfrac{V(T_1)+V(T_0)}{n}$ rooks, so $c_0+c_1=\dfrac{V(T_1)+V(T_0)}{n}$, and
\[
c_1=\dfrac{V(T_1)+V(T_0)}{n}-c_0.
\]
Let the brick $T_2$ be the auxiliary brick of $T_1$ along $t_2$ and let $T_2$ contain $c_2$ rooks. So $T_1\cup T_2$ is an $n$-brick, so $T_1\cup T_2$ has $\dfrac{V(T_2\cup T_1)}{n}=\dfrac{V(T_2)+V(T_1)}{n}$ rooks, so $c_1+c_2=\dfrac{V(T_2)+V(T_1)}{n}$, so
\begin{align*}
c_2&=\frac{V(T_2)+V(T_1)}{n}-c_1\\
&=\frac{V(T_2)+V(T_1)}{n}-\left(\frac{V(T_1)+V(T_0)}{n}-c_0\right)=
\frac{V(T_2)-V(T_0)}{n}+c_0.
\end{align*}
Let the brick $T_3$ be the auxiliary brick of $T_2$ along $t_3$ and let $T_3$ contain $c_3$ rooks. So $T_2\cup T_3$ is an $n$-brick, so $T_2\cup T_3$ has $\dfrac{V(T_3\cup T_2)}{n}=\dfrac{V(T_3)+V(T_2)}{n}$ rooks, so $c_2+c_3=\dfrac{V(T_3)+V(T_2)}{n}$, so
\begin{align*}
c_3&=\frac{V(T_3)+V(T_2)}{n}-c_2\\
&=\frac{V(T_3)+V(T_2)}{n}-\left(\frac{V(T_2)-V(T_0)}{n}+c_0\right)=\frac{V(T_3)+V(T_0)}{n}-c_0.
\end{align*}
When we take $T_{i+1}$, the auxiliary brick of $T_i$, and calculate $c_{i+1}$, then $V(T_i)$ falls out, the sign of $V(T_0)$ and $c_0$ change, in the next step $V(T_{i-1})$ falls out, the sign of $V(T_0)$ and $c_0$ change, etc\ldots
\begin{align*}
c_{i+1}&=\frac{V(T_{i+1})+V(T_i)}{n}-c_i=\frac{V(T_{i+1})+V(T_i)}{n}-\left(\frac{V(T_i)+V(T_{i-1})}{n}-c_{i-1}\right)=\ldots\\
&=\frac{V(T_{i+1})-(-1)^{i+1}V(T_0)}{n}+(-1)^{i+1}c_0.
\end{align*}
So \eqref{(201)} holds, and we can write it in the following form:
\begin{equation}\label{(202)}
(-1)^k\left[\frac{V(T_0)}{n}-c_0\right]=\frac{V(T_k)}{n}-c_k.
\end{equation}
\end{proof}

We also give a second, shorter proof, which yields the theorem in a more general
form.

\begin{proof}[Second proof]
For a cell $x\in H_n^d$ let $\chi(x)=1$ if $x$ contains a rook and $\chi(x)=0$
otherwise, and put
\[
\psi(x)=\frac1n-\chi(x).
\]
If $X$ is an arbitrary set of cells, then $V(X)/n=\sum_{x\in X}1/n$ and the number
of rooks in $X$ is $\sum_{x\in X}\chi(x)$, hence
\[
\df(X)=\sum_{x\in X}\psi(x).
\]
Let the subsystem $X$ be induced by the family $\langle E_1,E_2,\ldots,E_d\rangle$.
Then
\begin{equation}\label{(205)}
\df(X)=\sum_{x_1\in E_1}\ \sum_{x_2\in E_2}\cdots\sum_{x_d\in E_d}
\psi(x_1,x_2,\ldots ,x_d),
\end{equation}
so, for every fixed $i$, the right-hand side of \eqref{(205)} is an additive
function of the set $E_i$: if $E_i$ is the disjoint union of $E_i'$ and $E_i''$,
then the corresponding two values of $\df$ add up to $\df(X)$.

Assume that $E_i=\mathbb{Z}_n$ for some $i$, that is, $X$ is an $n$-brick. Fixing
the coordinates $x_j$ for $j\neq i$, the cells $(x_1,\ldots ,x_i,\ldots ,x_d)$ with
$x_i\in\mathbb{Z}_n$ form a file, which contains exactly one rook, so the inner sum
over $x_i$ equals $n\cdot\frac1n-1=0$. Consequently
\begin{equation}\label{(206)}
\df(X)=0\qquad\text{whenever $X$ is an $n$-brick.}
\end{equation}

Finally, let $\overline{E_i}=\mathbb{Z}_n\setminus E_i$. Since $E_i$ and
$\overline{E_i}$ partition $\mathbb{Z}_n$, additivity in the $i$-th variable and
\eqref{(206)} give
\[
\sum_{x_i\in\overline{E_i}}\ (\cdots)
=\sum_{x_i\in\mathbb{Z}_n}(\cdots)-\sum_{x_i\in E_i}(\cdots)
=0-\sum_{x_i\in E_i}(\cdots),
\]
that is, replacing $E_i$ by its complement changes the sign of $\df$. Replacing $k$
of the sets $E_1,E_2,\ldots ,E_d$ by their complements therefore multiplies $\df$
by $(-1)^k$, so
\[
\df(T_k)=(-1)^k\df(T_0),
\]
and \eqref{(201)} follows from $c_k=V(T_k)/n-\df(T_k)$.%
\end{proof}

\begin{rmrk}\label{rmrk:general}
The second proof uses neither the visibility of the subsystems nor their position on
the chess-board, so the Distribution Theorem~\ref{theo2.17} holds for an arbitrary
family $\langle E_1,E_2,\ldots ,E_d\rangle$ of subsets of $\mathbb{Z}_n$, and not
only for Hamming bricks. Bricks are used in the sequel because they admit a
convenient Euclidean picture, not because the result requires them.
\end{rmrk}

\begin{rmrk}
The only property of $\chi$ used in the second proof is that its sum along every
file equals $1$. Hence the proof applies verbatim if the rooks are replaced by the
entries of a $d$-fold stochastic matrix.
\end{rmrk}

\noindent
So, for a given $k$, $c_k$ depends only on $c_0$ and the volume of $T_0$ and $T_k$ and the parity of $k$.

\begin{defi}
The \emph{density} of a set of cells $X$ is $\varrho(X)=c/V(X)$, where $c$ is the number of rooks in $X$.
\end{defi}

In case of a $d$-LSC the density of the entire chess-board is: $\varrho(H_n^d)=n^{d-1}/n^d=1/n$.

\begin{defi}
The set of cells $X$ has \emph{standard density} if $\varrho(X)=1/n$.
\end{defi}

\begin{defi}
We define the \emph{deflection} of the set of cells $X$ from the standard density as follows
\[
\df(X) = V(X)/n - c
\]
where $c$ is the number of rooks in $X$.
\end{defi}

\begin{rmrk}\label{rmrk:discrepancy}
For a brick $T$ the deflection is, up to sign, the quantity studied under the name of
\emph{discrepancy} by Linial and Luria~\cite{[4]}. They consider boxes
$T=T_1\times\cdots\times T_d$ of a high-dimensional permutation, define
$\mathrm{vol}(T)=\prod_i|T_i|$, and measure the deviation of the number of ones in $T$
from $\mathrm{vol}(T)/n$; in our notation this deviation is exactly $-\df(T)$. Their
concern is how small the deviation can be made over \emph{all} boxes simultaneously.
Ours is different and complementary: we do not estimate $\df$, we show that within the
partition generated by a brick the deflections of the $2^d$ bricks are all equal up to
sign, and are therefore determined by any one of them.
\end{rmrk}

The set of cells $X$ has standard density if and only if $\df(X)=0$. Using the deflection, the equality in \eqref{(201)} can be written in the following two ways:

\begin{theo}[Deflection Theorem]
\begin{equation}\label{(203)}
c_k=\frac{V(T_k)}{n}-(-1)^k\df(T_0).
\end{equation}
\end{theo}

\begin{theo}[Main Theorem]
\begin{equation}\label{(204)}
\df(T_k)=(-1)^k\df(T_0).
\end{equation}
\end{theo}

The Main Theorem says more than that the two members of a remote couple are tied to
each other: it determines the whole partition generated by $T_0$ at once.

\begin{cor}[One degree of freedom]\label{cor:onedof}
Let $L$ be a $d$-LSC and let $T_0$ be a brick. Then the numbers of rooks in all the
$2^d$ bricks of the partition generated by $T_0$ are determined by the volumes of
these bricks together with the single number $\df(T_0)$, namely
\[
c_I=\frac{V(T_I)}{n}-(-1)^{|I|}\df(T_0)
\qquad\text{for every }I\subseteq\{1,2,\ldots ,d\}.
\]
In particular, knowing the number of rooks in one brick of the partition determines
it in every other brick of the partition.
\end{cor}

\begin{rmrk}
For $d=3$ and $I=\{1,2,3\}$ the corollary specialises to the identity
$c_0+c_3=n^2-(a+b+c)n+(ab+bc+ca)$ of Cruse~\cite{[2]}, which concerns a brick and
its remote mate only. Corollary~\ref{cor:onedof} covers every brick of the
partition, and it does so in every dimension.
\end{rmrk}

If $T_0$ has standard density, then $\df(T_k)=0$, ergo, each Hamming brick also has standard density. If $\df(T_0)\neq 0$, then each brick in the Hamming sphere $S_k(T_0)$ has the same deflection, either $\df(T_0)$ if $k$ is even or $-\df(T_0)$ if $k$ is odd. Ergo, the sign of the deflection of the bricks alternates when we step to the bricks of the next Hamming sphere $S_{k+1}(T_0)$. Note, that the deflection is not necessarily integer.

If $T_0$ is not a real brick, then $T_0$ has at least one edge of length $n$, so $\df(T_0)=0$, thus, $T_0$ has a standard density. Consequently, in this case, each brick of the partition has a standard density.

\begin{cor}
If $n$ is a prime number, then there is no real Hamming brick of standard density.
\end{cor}

\subsection*{Case $d=2$:}
The brick $T_0$, colored yellow in the Figure~\ref{fig2_3}, is a real brick of size $a\times b$. The bricks $T_{1a}$ and $T_{1b}$ 
are the two auxiliaries of $T_0$ and the brick $T_2$ is the remote mate of $T_0$.
Each file has exactly one rook and let $T_0$ have $c_0$ rooks.
Then $T_{1a}$ has
\[
\frac{V(T_0)+V(T_{1a})}{n}-c_0=b-c_0
\]
rooks and so $T_2$ has
\[
\frac{V(T_{1a})+V(T_2)}{n}-\left(\frac{V(T_0)+V(T_{1a})}{n}-c_0\right)=\frac{V(T_2)-V(T_0)}{n}+c_0
\]
rooks. So
\[
c_2=\frac{V(T_2)-V(T_0)}{n}+c_0=(n-a-b)+c_0.
\]
Another form of this equality is 
\[
c_0-c_2=a+b-n.
\]
\begin{figure}[htb]
\centering\includegraphics[scale=.4]{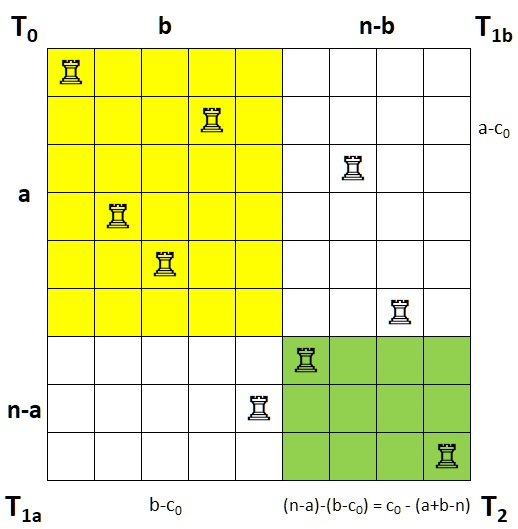}
\caption{}\label{fig2_3}
\end{figure}

\begin{defi}
If $T_0$ is a real brick of size $a\times b$, then the number $(a+b-n)$ is called the \emph{Ryser-number} of $T_0$ and is denoted by $\Ry(T_0)$.
\end{defi}

The Ryser-number of $T_0$ is the difference between the numbers of rooks in $T_0$ and $T_2$. If $c_0$ is known, then we know exactly how many rooks are in $T_2$, that is $c_2=c_0-\Ry(T_0)$.

It is clear, that $\Ry(T_2)=-\Ry(T_0)=(n-a-b)$ and $c_2=c_0+\Ry(T_2)$.

\begin{cor}
If $a+b=n$, then the white bricks  $T_{1a}$ and $T_{1b}$ are squares and $\Ry(T_0)=0$, so $T_0$ and $T_2$ have the same number of rooks.
\end{cor}

\subsection*{Case $d=3$:} 
The brick $T_0$, colored yellow in the Figure~\ref{fig2_4}, is a real brick of size $a\times b\times c$. The bricks $T_{1a}$, $T_{1b}$ and $T_{1c}$ are the auxiliaries of $T_0$, the brick $T_3$ is the remote mate of $T_0$ and the bricks $T_{2ab}$, $T_{2bc}$ and $T_{2ac}$ are the auxiliaries of $T_3$.

Each file has exactly 1 rook and let $T_0$ have $c_0$ rooks.
Then $T_0\cup T_{1a}$ has $\dfrac{V(T_0)+V(T_{1a})}{n}=bc$ rooks, so $T_{1a}$ has
\[
\frac{V(T_0)+V(T_{1a})}{n}-c_0=bc-c_0
\]
rooks.
\begin{figure}[htb]
\centering\includegraphics[scale=.5]{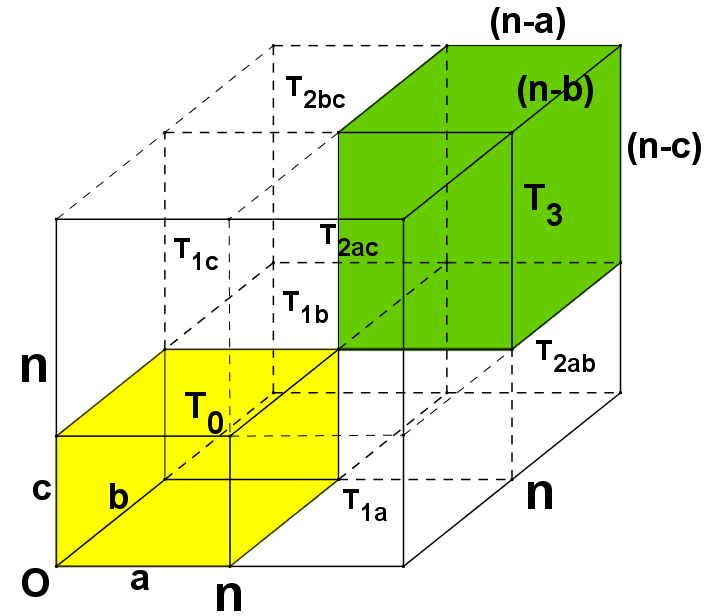}
\caption{}\label{fig2_4}
\end{figure}

\noindent
$T_{1a}\cup T_{2ab}$ has $\dfrac{V(T_{1a})+V(T_{2ab})}{n}=(n-a)c$ rooks, hence $T_{2ab}$ has
\[
\frac{V(T_{1a})+V(T_{2ab})}{n}-\left(\frac{V(T_0)+V(T_{1a})}{n}-c_0\right)=\frac{V(T_{2ab})-V(T_0)}{n}+c_0
\]
$T_{2ab}\cup T_3$ has $\dfrac{V(T_{2ab})+V(T_3)}{n}=(n-a)(n-b)$ rooks, so $T_3$ has
\[
\frac{V(T_{2ab})+V(T_3)}{n}-\left(\frac{V(T_{2ab})-V(T_0)}{n}+c_0\right)=\frac{V(T_3)+V(T_0)}{n}-c_0.
\]
As a result,
\begin{align*}
c_0+c_3&=\frac{V(T_0)+V(T_3)}{n}=\frac{abc+(n-a)(n-b)(n-c)}{n}\\
&=n^2-(a+b+c)n+(ab+bc+ca).
\end{align*}
If $c=0$ or $c=n$, then we also get the right numbers.

\begin{rmrk}
Let $M$ be a $d$-fold stochastic matrix of dimension $d$. Then the sum of the numbers in each file is exactly~1. Let $c_0$ be the sum of the numbers in $T_0$ and $c_k$ be the sum of the numbers in $T_k$. Then the Distribution Theorem~\ref{theo2.17} holds for $d$-fold stochastic matrix, you only have to change the expression from ``number of rooks” to ``sum of the numbers” in the proof. 
\end{rmrk}

\begin{cor}
The sum of the numbers in $T_k$ depends only on the sum of the numbers in $T_0$ and the volume of $T_0$ and $T_k$ and the parity of $k$. 
\end{cor}

\begin{rmrk}
The Distribution Theorem~\ref{theo2.17} holds for any types of numbers of $d$-fold stochastic matrices too, if the sum of the numbers is exactly~1 in each file, but we are only dealing with the non-negative case.
\end{rmrk}

\begin{rmrk}
For $d=3$ we get $c_0+c_3=n^2-(a+b+c)n+(ab+bc+ca)$, which is a result of Cruse~\cite{[2]} for triply stochastic matrices including permutation cubes, which are characteristic matrices of Latin squares. Consequently, for triply stochastic matrices $c_0+c_3$ always is an integer.
\end{rmrk}

\begin{rmrk}
The $d=2$ we get $c_2=(n-a-b)+c_0$, thus $c_0-c_2 = a+b-n$, consequently, for doubly stochastic matrices $c_2-c_0$ always is an integer.
\end{rmrk}

\begin{defi}
A \emph{file of brick} $T$ means the cells of a file, which are in $T$.
\end{defi}

\begin{defi}
A \emph{layer of brick} $T$ means the cells of a layer, which are in $T$.
\end{defi}

\begin{defi}
The brick $T$ is large, if $T$ has a layer of size $a\times b$, for which $a+b>n$.
\end{defi}

\begin{cor}
Let $T$ be a large brick in a $d$-fold stochastic matrix or in a $d$-LSC. Then $T$ has no layer, for which the sum of numbers or the number of rooks is $0$, respectively.
\end{cor}

\begin{cor}
If $T$ is a brick of size $a\times b\times c$ in a $3$-LSC and $a+b>n$, $b+c>n$, $a+c>n$, then each layer of the brick $T$ has at least one rook.
\end{cor}

\goodbreak
\section{Sum of Hamming Distances}

We first recall a classical identity for binomial coefficients, which will be
used below. It is well known; we include a short proof only to keep the paper
self-contained.

\begin{statm}\label{theo4.1}
\[
\sum_{k=0}^dk(-1)^k\binom{d}{k}=0
\]
for any $d\geq 2$.
\end{statm}

\begin{proof}
The term of index $k=0$ vanishes, so the sum may be started at $k=1$. For
$1\leq k\leq d$ the absorption identity
\[
k\binom{d}{k}=d\binom{d-1}{k-1}
\]
holds, hence, substituting $j=k-1$,
\[
\sum_{k=0}^dk(-1)^k\binom{d}{k}
=d\sum_{k=1}^d(-1)^k\binom{d-1}{k-1}
=-d\sum_{j=0}^{d-1}(-1)^j\binom{d-1}{j}
=-d(1-1)^{d-1}=0,
\]
where the last equality uses $d-1\geq 1$.%
\end{proof}

\begin{rmrk}
Equivalently, differentiating $\sum_{k=0}^d\binom{d}{k}(-x)^k=(1-x)^d$ with
respect to $x$ gives $\sum_{k=1}^dk(-1)^k\binom{d}{k}x^{k-1}=-d(1-x)^{d-1}$,
whose right-hand side vanishes at $x=1$ whenever $d\geq 2$.
\end{rmrk}

Let $L$ be a $d$-LSC and $T_0$ be a real brick. Each rook has a Hamming distance from $T_0$. Let us denote $s_k$ the number of rooks in the Hamming sphere $S_k(T_0)$, and $V(S_k(T_0))$ is the sum of the volumes of all the bricks of the Hamming sphere $S_k(T_0)$.

We sum up all Hamming distances from $T_0$ for all rooks. The sum is denoted by $h_n^d(T_0)$.

$S_k(T_0)$ has $\dbinom{d}{k}$ bricks, each of which has $c_k=\dfrac{V(T_k)}{n}-(-1)^k\df(T_0)$ rooks. Thus, the sum
\[
s_k=\dfrac{V(S_k(T_0))}{n}-\dbinom{d}{k}(-1)^k\df(T_0).
\]
\begin{align*}
h_n^d(T_0)&=\sum_{k=0}^dks_k=\sum_{k=0}^dk\left[\frac{V(S_k(T_0))}{n}-\binom{d}{k}(-1)^k\df(T_0)\right]\\
&=\sum_{k=0}^dk\frac{V(S_k(T_0))}{n}-\sum_{k=0}^dk\binom{d}{k}(-1)^k\df(T_0)\\
&=\sum_{k=0}^dk\frac{V(S_k(T_0))}{n}-\df(T_0)\sum_{k=0}^dk(-1)^k\binom{d}{k}.
\end{align*}
Because of 
Statement~\ref{theo4.1}
\begin{equation}\label{(401)}
h_n^d(T_0)=\sum_{k=0}^dks_k=\sum_{k=0}^dk\frac{V(S_k(T_0))}{n}
\end{equation}
\begin{theo}[Distance Theorem]
\begin{equation}\label{(402)}
h_n^d(T_0)=n^{d-2}\left[(n-e_1)+(n-e_2)+\ldots+(n-e_d)\right] 
\end{equation}
\end{theo}

\begin{proof}
We prove, that
\begin{equation}\label{(403)}
nh_n^d(T_0)=\sum_{k=0}^d k V(S_k(T_0))=n^{d-1}\left[(n-e_1)+(n-e_2)+\ldots+(n-e_d)\right]
\end{equation}
by induction on $d$.

In the case $d=2$, there are $c_0$ rooks with distance $0$, $c_{1a}+c_{1b}$ rooks with distance 1 and $c_2$ rooks with distance~2. The sum is:
\begin{align*}
h_n^2(T_0)&=0\cdot c_0+1\cdot (c_{1a}+c_{1b})+2\cdot c_2\\
&=(b-c_0)+(a-c_0)+2(c_0-a-b+n)=2n-a-b=(n-a)+(n-b),
\end{align*}
so
\[
nh_n^2(T_0)=n\big[(n-a) + (n-b)\big]
\]
Let us assume that $d>2$ and \eqref{(403)} holds.
\[
\sum_{k=0}^{d+1}kV(S_k^{d+1}(T_0))=n^d\big[(n-e_1)+(n-e_2)+\ldots+(n-e_d)+(n-e_{d+1})\big]
\]
Because of \eqref{(401)} 
\[
nh_n^{d+1}(T_0)=\displaystyle\sum_{k=0}^{d+1}kV(S_k^{d+1}(T_0))
\]

Let $T_0^{d+1}$ be a real brick in $H_n^{d+1}$. The volume of $T_0^{d+1}$ is $V(T_0^{d+1})=\displaystyle\prod_{i=1}^{d+1}e_i$. Let $T_k^{d+1}$ be a Hamming brick in the partition generated by $T_0^{d+1}$, and $d\big(T_0^{d+1}, T_k^{d+1}\big)=k$, then there exist an index set $I={i_1,i_2,\ldots,i_k}$ such, that we changed $e_i$ of $T_0^{d+1}$ to $(n-e_i)$ in $T_k^{d+1}$ for $i\in I$ and we kept $e_j$ of $T_0^{d+1}$ in $T_k^{d+1}$ for $j\notin I$. There are exactly $\dbinom{d+1}{k}$ index set with this property and each of them defines a $T_k^{d+1}$ brick.

The volume of $T_0^{d+1}$ is $V(T_0^{d+1})=\displaystyle\prod_{i=1}^{d+1}e_i$ and the volume of $T_k^{d+1}$ is
\[
V(T_k^{d+1})=\prod_{i\in I}(n-e_i)\cdot \prod_{j\notin I}e_j
\]

There are two types of bricks that have the distance $k$ from $T_0$. The first type has an edge $e_{d+1}$ on the axis $t_{d+1}$, the other has an edge $(n-e_{d+1})$ on the axis $t_{d+1}$. Take all bricks from the first type.

In this case we changed $k$ edges from the edges of $T_0^d$. The sum of the volume of these bricks is $e_{d+1}V(S_k(T_0^d))$. In the other case we changed $k-1$ edges from the edges of $T_0^d$ and the $e_{d+1}$ to $(n-e_{d+1})$.

The sum of the volume of these bricks is $(n-e_{d+1})V(S_{k-1}(T_0^d))$. So
\begin{align*}
V\big(S_k^{d+1}(T_0^{d+1})\big)&=e_{d+1}V\big(S_k^d(T_0^d)\big)+(n-e_{d+1})V\big(S_{k-1}^d(T_0^d)\big)\\
\sum_{k=0}^{d+1}kV\big(S_k^{d+1}(T_0^{d+1})\big)&=\sum_{k=0}^{d+1}k\left[e_{d+1}V(S_k^d(T_0^d))+(n-e_{d+1})V(S_{k-1}^d(T_0^d))\right]\\
&=e_{d+1}\sum_{k=0}^{d+1}kV(S_k^d(T_0^d))+(n-e_{d+1})\sum_{k=0}^{d+1}kV(S_{k-1}^d(T_0^d)).
\end{align*}
Because of $(d+1)V(S_{d+1}^d(T_0^d))=0$ and $0\cdot V(S_{-1}^d(T_0^d))=0$ we get
\begin{align*}
nh_n^{d+1}(T_0)&=\sum_{k=0}^{d+1}kV(S_k^{d+1}(T_0^{d+1}))\\
&=e_{d+1}\sum_{k=1}^{d+1}kV(S_k^d(T_0^d))+(n-e_{d+1})\sum_{k=1}^{d+1}kV(S_{k-1}^d(T_0^d))\\
&=e_{d+1}nh_n^d(T_0)+(n-e_{d+1})\left[\sum_{k=0}^d1\cdot V(S_k^d(T_0))+\sum_{k=0}^d kV(S_k^d(T_0))\right]\\
&=e_{d+1}nh_n^d(T_0)+(n-e_{d+1})\left[n^d+nh_n^d(T_0)\right]\\
&=e_{d+1}nh_n^d(T_0)+n\left[n^d+nh_n^d(T_0)\right]-e_{d+1}\left[n^d+nh_n^d(T_0)\right]\\
&=e_{d+1}nh_n^d(T_0)+n^{d+1}+nnh_n^d(T_0)-e_{d+1}n^d-e_{d+1}nh_n^d(T_0)\\
&=(n-e_{d+1})n^d+ n^2h_n^d(T_0).
\end{align*}
Based on the assumption of the induction we change $h_n^d(T_0)$ for the right part of the equality \eqref{(402)}
\begin{align*}
nh_n^{d+1}(T_0)&=(n-e_{d+1})n^d+n^2n^{d-2}\left[(n-e_1)+(n-e_2)+\ldots+(n-e_d)\right]\\
&=n^d\left[(n-e_1)+(n-e_2)+\ldots+(n-e_d)+(n-e_{d+1})\right]
\end{align*}
so
\begin{equation}\label{(404)}
\begin{split}
h_n^{d+1}(T_0)&=n^{d-1}\left[(n-e_1)+(n-e_2)+\ldots+(n-e_d)+(n-e_{d+1})\right]\\
&=n^{d-1}\vcount(T_{d+1})
\end{split}
\end{equation}
So \eqref{(402)} holds for any $d>1$.

If $T_0$ is not a real brick, then $T_0$ is an $n$-brick. Since $\df(T_0)=0$, so \eqref{(402)} holds. Let's assume, that $T_0$ has $m$ edges, $e_1$, $e_2$, \ldots, $e_m<n$. Every $T_k$ brick in the partition generated by $T_0$ has $d-m$ edges of length $n$. So for $T_k$ of dimension $d$ the $V(T_k^d)=n^{d-m}V(T_k^m)$, where $T_k^m$ means a brick of dimension $m$ and edges $e_1$, $e_2$, \ldots, $e_m$. In this case \eqref{(402)} holds for $T_0^m$, so $h_n^m(T_0^m)=n^{m-2}\left[(n-e_1)+(n-e_2)+\ldots+(n-e_m)\right]$
\begin{align*}
n^{d-m} h_n^m(T_0^m)&=n^{d-2}\left[(n-e_1)+(n-e_2)+\ldots+(n-e_m)\right.\\
&\qquad\left.{}+(n-n)+\ldots+(n-n)\right]
\end{align*}
The left-hand side is $h_n^d(T_0)$, therefore, \eqref{(402)} holds for not real bricks too.%
\end{proof}

\begin{cor}
If we take the sum of the Hamming distances of all rooks of a $d$-LSC from a brick $T_0$, then the result does not depend on the number of rooks in $T_0$.
\end{cor}

\begin{cor}
If $T_0$ is a real brick, then $T_d$ is the only brick that has a distance $d$ from $T_0$, so because of \eqref{(402)}
\begin{equation}\label{(405)}
\begin{aligned}
h_n^d(T_0)&=n^{d-2}\left[(n-e_1)+(n-e_2)+\ldots+(n-e_d)\right]=n^{d-2}\vcount(T_d)\\
h_n^d(T_d)&=n^{d-2}\left[e_1+e_2+\ldots+e_d\right]=n^{d-2}\vcount(T_0) 
\end{aligned}
\end{equation}
\end{cor}

\section{Remote Brick Couples (RBCs)}

From now on we are dealing with the cases $d=2$ and $d=3$.
Based on the Distribution Theorem~\ref{theo2.17} 
\begin{align*}
c_0-c_2&=\frac{V(T_2)-V(T_0)}{n}=\frac{ab-(n-a)(n-b)}{n}=a+b-n\qquad\text{and}\\
c_0+c_3&=\frac{V(T_3)+ V(T_0)}{n}=\frac{abc+(n-a)(n-b)(n-c)}{n}\\
&=n^2-(a+b+c)n+(ab+bc+ca)
\end{align*}

\begin{defi}
For $d=2$ the RBC $(T_0,T_2)$ is called \emph{balanced}, if $c_0-c_2=a+b-n$ holds.
\end{defi}

\begin{defi}
For $d=3$ let the \emph{capacity} of RBC $(T_0,T_3)$ be 
\[
\capa(T_0,T_3)=n^2-(a+b+c)n+(ab+bc+ca)
\]
\end{defi}

\begin{defi}
For $d=3$ the RBC $(T_0,T_3)$ is called \emph{stuffed}, if $c_0+c_3=\capa(T_0,T_3)$.
\end{defi}

\begin{cor}
For a 2-LSC $L$ holds, that each RBC $(T_0,T_2)$ of $L$ is balanced.
For a 3-LSC $M$ holds, that each RBC $(T_0,T_2)$ of any layer of $M$ is balanced and each RBC $(T_0,T_3)$ of $M$ is stuffed.
\end{cor}

\begin{rmrk}
$T_0$ has different meanings for $d=2$ and $d=3$.
\end{rmrk}

\begin{rmrk}
The expression $\capa(T_0,T_3)$ is referred sometimes by $\capa(n,a,b,c)$ to emphasize that $\capa(T_0,T_3)$ is a function with integer variables $a$, $b$, $c$, where $a$, $b$, $c\in\{1,2\ldots,n\}$ and the integer $n$ is fixed. Cruse~\cite{[2]} proved some properties of this function.
\end{rmrk}

\begin{defi}
The brick $T$ is \emph{degenerated}, if at least one edge of $T$ is~$0$.
\end{defi}

If $T$ is degenerated, then it contains no rooks.

We consider the remote bricks $(T_0,T_3)$ as degenerated RBC, if either $T_0$ or $T_3$ has one edge of size~0, for example $c = 0$ or $n-c = 0$. If $c=0$ then the size of $T_0$ is $a\times b\times 0$, if $n-c=0$ then the size of $T_3$ is $(n-a)\times (n-b)\times 0$ as shown in Figure~\ref{fig5_1}.
The capacity function provides the correct results for both cases, the number of rooks that the $n$-brick can contain.

\begin{figure}[htb]
\centering
\hbox{\includegraphics[scale=.234]{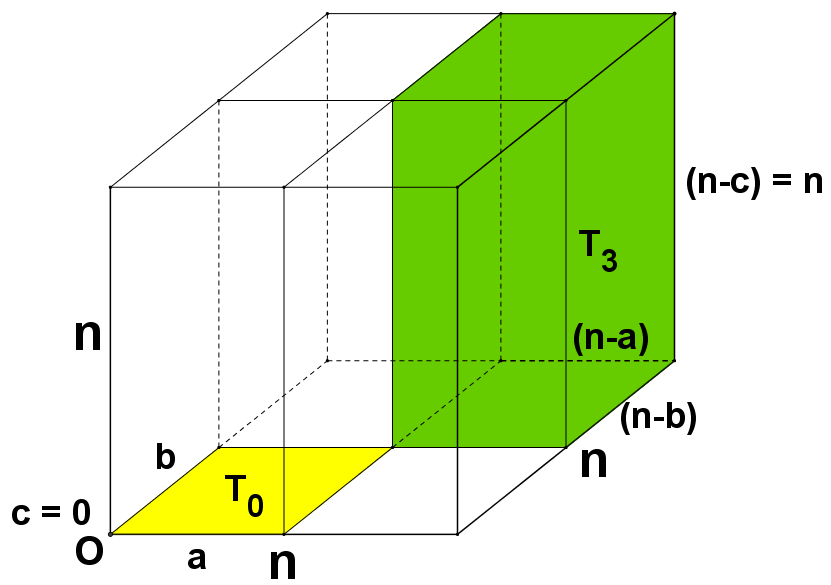}\hfill\includegraphics[scale=.234]{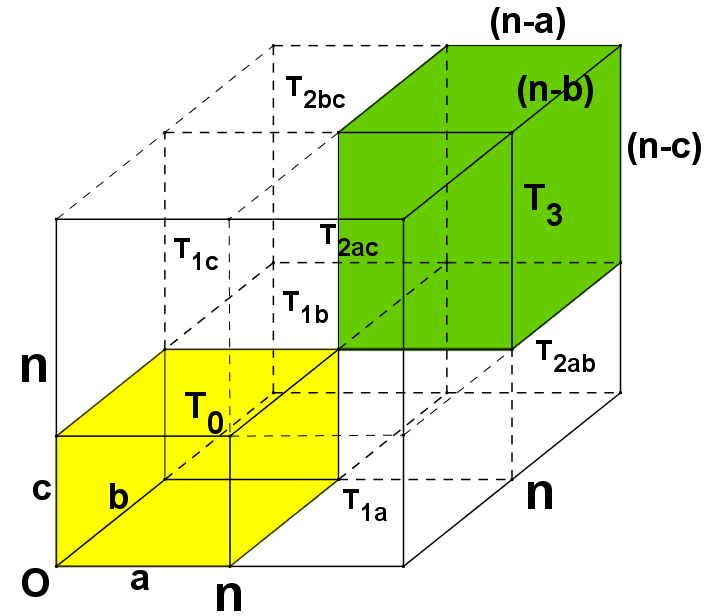}\hfill\includegraphics[scale=.234]{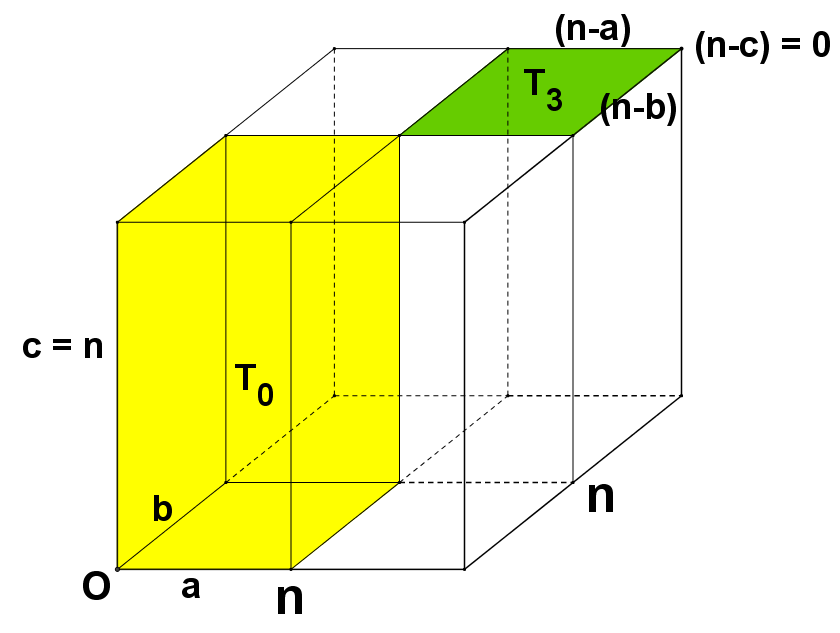}}
\caption{}\label{fig5_1}
\end{figure}

The capacity function can be written as follows:
\[
\capa(n,a,b,c)=n^2-(a+b+c)n+(ab+bc+ca)=(n-a)(n-b)+c(a+b-n).
\]
If the integer variable $c$ goes from $0$ to $n$, then the value of the function goes from $(n-a)(n-b)$, which is the area of $T_3$ orthogonal to $z$, to $ab$, which is the area of $T_0$ orthogonal to $z$, as depicted on the left-hand side of Figure~\ref{fig5_1} and on the right-hand side of Figure~\ref{fig5_1}, respectively. The change is $(a+b-n)$ in each step.

The layer $c$ is balanced, so the yellow brick of the layer has $c_0$ rooks and the remote brick in this the layer has $c_2$ rooks and $c_0-c_2=(a+b-n)$.

The step from $(c-1)$ to $c$ can be seen in the Figure~\ref{fig5_2}.
\begin{figure}[htb]
\centering
\includegraphics[scale=.45] {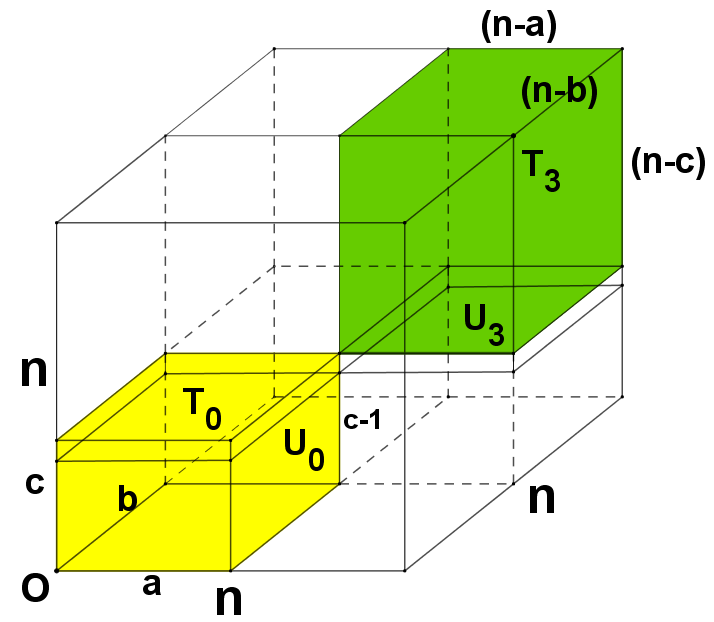}
\caption{}\label{fig5_2}
\end{figure}

After this step, the number of rooks in brick $T_3$ changes by $-c_2$, the number of rooks in brick $T_0$ changes by $c_0$, so by $c_0-c_2=(a+b-n)$ combined.

If $a+b-n>0$, then $(n-a)(n-b)<ab$, so the capacity increases from $(n-a)(n-b)$ to $ab$, if $a+b-n<0$, then $(n-a)(n-b)>ab$, so the capacity decreases from $(n-a)(n-b)$ to $ab$ and if $a+b-n=0$, then $(n-a)(n-b)=ab$, so the capacity does not change, i.e, the RBCs have the same capacity for all~$c$.

The capacity of an RBC can be regarded as a generalization of the capacity of an $n$-brick (degenerated RBC).

\begin{rmrk}
In a $3$-LSC $c_0+c_3=\dfrac{V(T_3)+V(T_0)}{n}$ for each RBC $(T_0,T_3)$. \\$V(T_3)+V(T_0)=V(T_3\cup T_0)$ because of $T_0\cap T_3=\emptyset$, hence
\[\varrho(T_0\cup T_3)=\dfrac{c_0+c_3}{V(T_0\cup T_3)}=\dfrac{c_0+c_3}{V(T_0)+V(T_3)}=\dfrac{1}{n},\] 
ergo, each RBC $(T_0,T_3)$ has a standard density.
\end{rmrk}

\begin{figure}[htb]
\centering
\includegraphics [scale=.5]{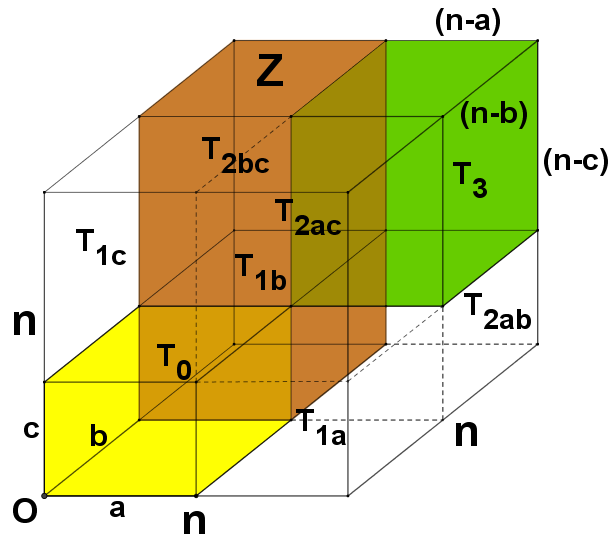}
\caption{}\label{fig5_3}
\end{figure}

\begin{defi}
An $n$-brick is called an \emph{axis} if it has exactly one edge of size $n$.
\end{defi}

If we consider the bricks $T_{1b}$ and $T_{2bc}$ together, like the brown $n$-brick in the Figure~\ref{fig5_3}
, then the structure of $T_0$, $T_{1b}$, $T_{2bc}$, $T_3$ combined looks like a hinge (door hinge).
The brown $n$-brick $T_{1b}\cup T_{2bc}$ is the axis of the hinge, the bricks $T_0$ and $T_3$ are the leafs of the hinge.

On one hand, a hinge can be considered as an axis with two leafs, on the other hand, a hinge is the union of two disjoint $n$-bricks $(T_0\cup T_{1b})$ and $(T_3\cup T_{2bc})$. That gives

\begin{cor}[Hinge Volume]
\begin{equation}\label{(502)}
V(T_0\cup T_3)=V(T_0\cup T_{1b})+(T_3\cup T_{2bc})-V(T_{1b}\cup T_{2bc})
\end{equation}
\end{cor}
and dividing both sides of \eqref{(502)} by $n$ we get

\begin{cor}[Hinge Capacity]\label{theo5.10}
\begin{equation}\label{(503)}
\capa(T_0,T_3)=\frac{V(T_0)+ V(T_{1b})}{n}+\frac{V(T_3)+ V(T_{2bc})}{n}-\frac{V(T_{1b})+V(T_{2bc})}{n} 
\end{equation}
\end{cor}

With the help of equality \eqref{(503)} we can prove, without using the Distribution Theorem~\ref{theo2.17}, that any RBC has as many rooks as its capacity.
\begin{theo}
For any RBC $(T_0,T_3)$ in an LSC
\[
c_0+c_3=n^2-(a+b+c)n+(ab+bc+ca)
\]
\end{theo}

\begin{proof}
The $n$-brick $T_0\cup T_{1b}$ contains $ac$ rooks, the $n$-brick $T_{2bc}\cup T_3$ contains $(n-b)(n-c)$ rooks and the $n$-brick $T_{1b}\cup T_{2bc}$ contains $a(n-b)$ rooks. So the RBC $(T_0,T_3)$ contains
\[
ac+(n-b)(n-c)-a(n-b)
\]
rooks, i.e,
\[
c_0+c_3=ac+(n-b)(n-c)-a(n-b),
\]
but the right-hand side is equal to $n^2-(a+b+c)n+(ab+bc+ca)$.%
\end{proof}

\begin{rmrk}
The complement of a hinge is also a hinge, so the chess-board $H^3$ consists of two disjoint hinges.
\end{rmrk}

\bibliographystyle{plain}

\end{document}